\documentclass{amsart}
\usepackage{amssymb}
\usepackage[colorlinks,backref,bookmarks]{hyperref}

\newtheorem{lemma}{Lemma}[section]
\newtheorem{theorem}[lemma]{Theorem}

\newtheorem{corollary}[lemma]{Corollary}

\theoremstyle{definition}
\newtheorem{example}[lemma]{Example}

\newtheorem{definition}[lemma]{Definition}

\newtheorem{fact}[lemma]{Fact}

\DeclareMathOperator{\Gl}{Gl}
\DeclareMathOperator{\F}{\mathbb{F}_q}
\DeclareMathOperator{\VI}{\mathcal{VI}}

\DeclareMathOperator{\Inj}{Inj}

\DeclareMathOperator{\CC}{\mathcal{C}}
\DeclareMathOperator{\Mon}{Mon}

\title{Dimension-independent statistics of $\Gl_n(\F)$ via character polynomials}
\author{Nir Gadish}
\address{Department of Mathematics, University of Chicago, Chicago, IL, USA}
\email{\href{mailto:nirg@math.uchicago.edu}{nirg@math.uchicago.edu}}

\begin{document}
	\begin{abstract}
		Picking permutations at random, the expected number of $k$-cycles is known to be $1/k$ and is, in particular, independent of the size of the permuted set. This short note gives similar size-independent statistics of finite general linear groups: ones that depend only on small minors. The proof technique uses combinatorics of categories, motivated by representation stability, and applies simultaneously to symmetric groups, finite linear groups and many other settings.
	\end{abstract}
	
	\maketitle
\vspace{-1em}
\section{Introduction}
It is well known that the expected number of fixed points of a random permutation is $1$, independent of the size of the permuted set. Perhaps less known is that the same is true for finite linear groups: the expected number of non-zero fixed vectors of a random matrix $T\in \Gl_n(\F)$ is $1$, independent of the dimension $n$ (see Corollary \ref{cor:fixed} below). This short note will demonstrate how the combinatorics of a category gives rise to a wide class of similar statistical invariants, whose moments are eventually independent of dimension. Similar considerations apply in vastly different contexts, giving `size'-independent statistics for natural sequences of groups, see \S3.

This article sets out to illuminate some of the properties and applications of the so-called \emph{generalized character polynomials} introduced in \cite{Ga} by specializing them to the specific case of finite general linear groups. Character polynomials are an algebra of class functions, defined simultaneously on an infinite collection of groups, and whose statistical properties relate the combinatorics and representation theory of the different groups -- hence they appear in the context of representation stability. It is the author's hope that the example will demonstrate the usefulness of this new character-theoretic tool, as well as confirming the efficacy of categorical considerations in combinatorics.

Statistics of finite matrix groups is a rich field with many effective techniques and applications to number theory, combinatorics and computer science (see e.g. \cite{Fulman} and the references therein). A typical question that one asks in this field is ``what is the probability that the characteristic polynomial of a randomly chosen matrix have a certain form?" and a typical answer is asymptotic in the size of the matrix. The kind of question that this article considers is different: while the characteristic polynomial depends on the entire matrix, we will focus on more local properties -- ones that depend only on small minors -- and our answers will be exact.


Let us motivate the invariants that shall be considered. Apart from the question of fixed points of a random permutation $\sigma\in S_n$, one can ask for the expected number of $d$-cycles, and the answer is again independent on $n$. This could be abstractly rephrased as the number of $\sigma$-invariant subsets of size $d$, to which $\sigma$ restricts to a $d$-cycle. A natural generalization to the setting of $\Gl_n(\F)$ is given as follows.
\begin{theorem}[\textbf{Stable statistics for $\Gl_\bullet(\F)$}]\label{thm:main}
	Fix a conjugacy class $C\subset  \Gl_d(\F)$. Then for a random $T\in \Gl_n(\F)$, the expected number of $d$-dimensional subspaces $W\leq \F^n$ for which $T(W)=W$ and $T|_W \in C$ is independent of $n$ once $n\geq d$. In particular, calculating the case $n=d$ gives this stable expectation as $\frac{|C|}{|\Gl_d(\F)|}$.
\end{theorem}
More generally, all joint higher moments of these random variables are eventually independent of $n$. Explicitly, define a random variable for every conjugacy class $C\subset  \Gl_d(\F)$ by
$$
X_C(T) = \#\{ W\leq \F^n \mid \dim W = d,\, T(W)=W \text{ and } T|_W\in C \}
$$
well-defined simultaneously on all $\Gl_n(\F)$.
\begin{theorem}\label{thm:products}
Let $X_{C_i}$ be the random variables corresponding to conjugacy classes $C_i\subset \Gl_{d_i}(\F)$ respectively. Then the expectation $\mathbb{E}_{\Gl_n}\left[ X_{C_1}\cdot\ldots \cdot X_{C_r} \right]$ is the same for all $n\geq d_1+\ldots+d_r$.
\end{theorem}

In particular, taking $C = \{1\} \subset \Gl_1(\F)$ gives a count of the number of fixed vectors $\neq 0$ of a random $T\in \Gl_n$. As the expected number of such is independent of $n$, it can be computed with $n=1$: picking $\lambda\in \F^\times$ randomly, there are $q-1$ non-zero fixed points if $\lambda= 1$ and $0$ otherwise.
\begin{corollary}\label{cor:fixed}
	The expected number of non-zero fixed vectors of a random $T\in \Gl_n(\F)$ equals $1$. The same is true when replacing ``fixed vectors" by ``eigenvectors with eigenvalue $\lambda\in \F^*$".
\end{corollary}
We remark that this corollary is a simple special case of the extensive calculation of Fulman-Stanton \cite{FS}, who compute all moments of this random variable using (complicated) generating functions. Their calculated moments are are indeed eventually independent of $n$, as predicted by Theorem \ref{thm:main}.

The proof of Theorem \ref{thm:main} given below uses the category of finite-dimensional $\F$-vector spaces, of which $\Gl_n(\F)$ are the automorphism groups. The reader will notice that the approach has nothing to do with linear algebra, and could be used to prove analogous results in vastly different contexts. For example, the same technique gives the analogous:
\begin{fact}[{\cite[Proposition 3.9]{CEF}}]
	Fix a conjugacy class $C \subset S_d$. Then for a random $\tau\in S_n$, the expected number of subsets $W\subset [n]$ of cardinality $d$ such that $\tau(W)=W$ and $\tau|_W \in C$ does not depend on $n$ once $n\geq d$. In particular, this expected number is $\frac{|C|}{d!}$.
\end{fact}
In \cite{CEF} this fact is proved combinatorially -- by counting the number of permutations of various kinds. It is somewhat comforting that a single argument produces the same result in the general linear setting as well as for permutations.

\section{Character polynomials for $Gl_\bullet(F_q)$ and proofs}
In \cite{Ga} the author attached a collection of class functions to a category satisfying a list of axioms, associated with projective representations of the category. It is further shown there that character inner products between these class functions stabilize. An example of a category to which the theory applies is $\VI$: of finite dimensional vector spaces over $\F$ and injective linear maps between them.

In Theorem \ref{thm:main} we make the observation that the stable character inner products for $\VI$ could alternatively be viewed as stabilizing statistics of the groups $\Gl_n(\F)$ themselves, and interpret the stabilization result explicitly in linear algebraic terms.

\begin{proof}[Proof of Theorem \ref{thm:main}]
Let us unpack the definitions in \cite{Ga}, specialized to the category of finite dimensional vector spaces over $\F$ and linear injections. For every two vector spaces $W,V$ consider the `binomial set'
$$
\binom{V}{W} := \Inj(W,V){/\Gl(W)}
$$
of subspaces in $V$ that are isomorphic to $W$. Here $\Inj(-,-)$ denotes the set of injections between two vector spaces, with its natural left (right) action of $\Gl(-)$ be post (pre-) composition. This quotient set is a natural extension of the classical binomial coefficient $\binom{n}{k}$. Denote the equivalence class of the injection $\iota$ by $[\iota]$. 

A transformation $T\in \Gl(V)$ fixes $[\iota]$ iff $T\circ \iota = \iota\circ S$ for some (unique) $S\in \Gl(W)$. For every conjugacy class $C\subset \Gl(W)$ we set
\begin{equation}\label{eq:X_definition}
X_C(T) = \#\left\{ [\iota]\in \binom{V}{W}^T \,\middle\vert\, T\circ \iota = \iota\circ S \text{ and } S\in C \right\}
\end{equation}
and declare that it has degree $\dim W$. These functions and their linear combinations, defined on all $\Gl(V)$ at the same time, are the specialization to the category $\VI$ of the notion to which \cite{Ga} refers as \emph{(generalized) character polynomials}.

It is clear that the set $\binom{V}{W}$ is the Grassmannian of $\F$-subspaces $W'\leq V$ of dimension $=\dim{W}$ along with its natural $\Gl(V)$-action. Furthermore, the relation $T\circ \iota = \iota\circ S$ amounts to saying that $T|_{W'}$ is conjugate to the transformation $S$ under the identification $\iota: W\rightarrow W'$. One sees that the function $X_C$ thus defined is the one presented in the introduction.

It is shown in \cite[Corollary 4.6]{Ga} that $\mathbb{E}_{\Gl_n}[X_C]$ does not depend on $n$ once $n\geq \deg(X_C)$. Theorem \ref{thm:main} follows.
\end{proof}

\begin{proof}[Proof of Theorem \ref{thm:products}]
    Now consider joint higher moments. It is shown in \cite[Corollary 3.9]{Ga} that the product of character polynomials is again a character polynomial, and that the product respects degree. Explicitly, if $X_{C_i}$ are character polynomials of respective degree $d_i$ for $i=1,2$, then 
    \begin{equation} \label{eq:expansion}
        X_{C_1}\cdot X_{C_2} = \sum_{j=1}^r \lambda_j X_{D_j}
    \end{equation}
    for scalars $\lambda_j\in \mathbb{C}$ and with every $X_{D_j}$ of degree $\leq d_1+d_2$. The result then follows by applying Theorem \ref{thm:main} to this linear combination.
\end{proof}
We remark that the expansion coefficients in Equation \eqref{eq:expansion} are encoded in the combinatorics of pullbacks and pushouts in the category $\VI$, as is the case for generalized character polynomials of any category (for full details see \cite[\S3.1]{Ga}). Computing these coefficients in general poses a nontrivial combinatorial problem already in the case of $\Gl_\bullet(\F)$ and of symmetric groups. For example, if we denote a conjugacy class by a matrix in Jordan form
    \[
    X_{(\lambda)}\cdot X_{(\mu)} = X_{\left(\begin{smallmatrix}\lambda & 0 \\ 0 & \mu \end{smallmatrix}\right)} \;\text{ and }\; {X_{(\lambda)}}^2 = X_{(\lambda)} + (q+1)qX_{\left(\begin{smallmatrix}\lambda & 0 \\ 0 & \lambda \end{smallmatrix}\right)}
    \]
    for $\lambda\neq \mu \in \F^{*}$.

\section{Generalizations}
    This paper interprets the techniques of \cite{Ga} as tools showing how the combinatorics of categories sheds light on stable statistics of their automorphism groups. Equi-expectation results similar to Theorem \ref{thm:main} hold in the following general context. Let $\CC$ be a locally-finite category (i.e. hom-sets all are finite). For every two objects $c$ and $d$ denote the monomorphisms from $c$ to $d$ by $\Mon(c,d)$ and the automorphism group of $d$ by $G_d$.
	\begin{definition}
	A \emph{$c$-shaped subobject} in $d$ is an orbit of $\Mon(c,d)/G_c$. The set of these is denoted by $\binom{d}{c}$, and comes with a natural $G_d$-action.
	\end{definition}

	\begin{theorem}[\textbf{Object independent statistics}]\label{thm:generalization}
		Suppose that $c$ and $d$ are two objects in a locally-finite category for which the composition action $G_d\curvearrowright \Mon(c,d)$ is transitive.
		
		Fix a conjugacy class $C\subset G_c$. Then picking a random $g\in G_d$ uniformly, the expected number of $c$-shaped subobjects $[f]\in \binom{d}{c}$ that are fixed by $g$ and on which $g\circ f = f\circ h$ for $h\in C$ is precisely $|C|/G_c$. In particular this expected number does not depend on $d$.
	\end{theorem}
	\begin{proof}
        The set $X:= \Mon(c,d)$ of monomorphisms, with its two group actions of $G_c$ and $G_d$ by (pre-)post-composition gives a procedure for promoting $G_c$-representations to ones of $G_d$, say over $\mathbb{C}$
        \[
        (G_c\curvearrowright M) \longmapsto G_d\curvearrowright \mathbb{C}[X]\otimes_{G_c} M.
        \]
        A straightforward calculation shows that the character of such an induced $G_d$-representation is
        \[
        \sum_{C\in \operatorname{conj}(G_c)} \chi_M(C) X_C
        \]
        where $X_C$ are the obvious generalization of \eqref{eq:X_definition} to this case.

        Since the action $G_d$ on $X$ is assumed to be transitive, one gets the coinvariant quotient
        \[
            \left( \mathbb{C}[X]\otimes_{G_c} M \right)_{G_d} \cong \mathbb{C}[G_d\backslash X]\otimes_{G_c} M = \mathbb{C}\otimes_{G_c} M \cong M_{G_c}.
        \]
        It is then classical that a character average $\frac{1}{|G|}\sum_{g\in G}\chi_V(g)$
        computes the dimension of the coinvariant quotient $V_G$. Thus starting with any $G_c$-representation $M$, the corresponding character averages are the same for $G_d$ and for $G_c$. Orthogonality of the character values of $G_c$ implies that the same equality already holds for every function $X_C$ separately. This is the claim we set out to prove.
	\end{proof}
	\begin{example}[\textbf{Finite symplectic groups}]
        Theorem \ref{thm:generalization} applies to the category of finite dimensional symplectic $\F$-vector spaces and linear isometries. One gets the same conclusion of Theorem \ref{thm:main}, but with the groups $\Gl_n(\F)$ replaced by $\operatorname{Sp}_{2n}(\F)$ and with $W\leq \F^{2n}$ ranging over symplectic subspaces.
        
        Note that Theorem \ref{thm:products} is not expected to hold, since the category in question does not have pullbacks (e.g. intersections of symplectic subspaces might not be symplectic), which was a necessary ingridient in \cite[Corollary 3.9]{Ga}.
	\end{example}

\end{document}